\documentclass[11pt, a4paper]{article}
\usepackage{amsmath, amsthm, amssymb, url, mathtools, color}

\newcommand{\ad}{\mathrm{ad}}
\newcommand{\supp}{\mathrm{supp}}
\newcommand{\Aut}{\mathrm{Aut}}
\newcommand{\ZZ}{\mathbb{Z}}
\newcommand{\Spec}{\mathrm{Spec}}
\newcommand{\Vir}{\mathrm{Vir}}
\newcommand{\diag}{\mbox{diag}}
\newcommand{\ch}{\mathrm{char}}
\newcommand{\0}{{\bf 0}}
\newcommand{\1}{{\bf 1}}

\theoremstyle{plain}
\newtheorem{theorem}{Theorem}[section]
\newtheorem{theorem*}{Theorem}
\newtheorem{corollary*}{Corollary}

\newtheorem{corollary}[theorem]{Corollary}
\newtheorem{lemma}[theorem]{Lemma}
\newtheorem{proposition}[theorem]{Proposition}
\newtheorem*{claim*}{Claim}

\theoremstyle{definition}
\newtheorem{definition}[theorem]{Definition}
\newtheorem{definition*}{Definition}
\newtheorem*{notation}{Notation}

\newtheorem{example}[theorem]{Example}
\newtheorem{remark}[theorem]{Remark}

\renewcommand{\phi}{\varphi}

\title{Code algebras, axial algebras and VOAs}
\author{Alonso Castillo-Ramirez\footnote{Departamento de Matematicas, Centro Universitario de Ciencias Exactas e Ingenierias, Universidad de Guadalajara, Mexico, email: alonso.castillor@academicos.udg.mx
}, Justin M\textsuperscript{c}Inroy\footnote{Heilbronn Institute for Mathematical Research, School of Mathematics, University of Bristol, University Walk, Bristol, BS8 1TW, UK, email: justin.mcinroy@bristol.ac.uk} and Felix Rehren}

\begin{document}
\maketitle

\begin{abstract}
Inspired by code vertex operator algebras (VOAs) and their representation theory, we define code algebras, a new class of commutative non-associative algebras constructed from binary linear codes. Let $C$ be a binary linear code of length $n$. A basis for the code algebra $A_C$ consists of $n$ idempotents and a vector for each non-constant codeword of $C$. We show that code algebras are almost always simple and, under mild conditions on their structure constants, admit an associating bilinear form. We determine the Peirce decomposition and the fusion law for the idempotents in the basis, and we give a construction to find additional idempotents, called the $s$-map, which comes from the code structure.  For a general code algebra, we classify the eigenvalues and eigenvectors of the smallest examples of the $s$-map construction, and hence show that certain code algebras are axial algebras.  We give some examples, including that for a Hamming code $H_8$ where the code algebra $A_{H_8}$ is an axial algebra and embeds in the code VOA $V_{H_8}$.
\end{abstract}


\section{Introduction}

Vertex operator algebras (VOAs) were first considered by physicists in connection with chiral algebras and 2D conformal field theory, and subsequently by mathematicians who noticed intriguing links between finite simple groups and modular functions, two apparently unrelated mathematical objects.  Essentially, a VOA is an infinite dimensional graded algebra with infinitely many different products which are linked in an intricate way. The prototypical example is the moonshine VOA, denoted by $V^\natural$, which has the Griess algebra as the weight two graded part and the Monster sporadic simple group $M$ as the automorphism group. Despite their relevance, VOAs are still mysterious objects: they have a deep theory and are quite difficult for explicit calculations.

Introduced by Hall, Rehren and Shpectorov in \cite{Axial1}, and extending earlier work by Ivanov on Majorana algebras \cite{I09}, axial algebras provide an axiomatic approach to better understand some important properties of VOAs. An \emph{axial algebra} is a commutative non-associative algebra generated by semi-simple primitive idempotents (i.e.\ the adjoint action decomposes the algebra as a direct sum of eigenspaces and the $1$-eigenspace is one-dimensional). Furthermore, the eigenvectors for any of the given idempotents multiply together in a specific way as given by a table called the \emph{fusion law}. Majorana algebras are a special case of axial algebras, directly linked with the Griess algebra and $V^\natural$, in which all idempotents in the generating set have eigenvalues $1$, $0$, $\frac{1}{4}$ and $\frac{1}{32}$, and satisfies a specific fusion law. Many axial algebras (Majorana algebras in particular) also admit a symmetric bilinear form that associates with the algebra product (i.e.\ $(v, u\cdot w) = (v \cdot u ,w)$ for all $v,u,w$); this is called a \emph{Frobenius form}.

Inspired by the axiomatic approaches to VOAs described above, we introduce \emph{code algebras}, a new class of commutative non-associative algebras constructed from binary linear codes.  Our construction is an axiomatisation of the construction of code VOAs. These were first studied by Miyamoto in \cite{M2, M3}, and by Dong, Griess and H\"ohn in \cite{DGH}, and they form an important class of VOAs whose representation theory is governed by two binary linear codes. In \cite{Lya}, Lam and Yamauchi show that every framed VOA $V$ (such as $V=V^\natural)$ has a uniquely defined code sub VOA and $V$ is a simple current extension of its code sub VOA. Moreover, Miyamoto provides a new construction of $V^\natural$ in this way in \cite{M4}. Every code VOA has a code algebra embedded in it, however code algebras are a wider class of algebras than those embeddable in code VOAs.

\begin{definition*}\label{CodeAlgebra}
Let $C \subseteq \mathbb{F}_2^n$ be a binary linear code of length $n$, $\mathbb{F}$ a field and $\Lambda$ be a collection of \emph{structure parameters}
\[
\Lambda := \left\{ a_{i,\alpha}, b_{\alpha,\beta}, c_{i,\alpha} \in \mathbb{F}  : i \in \supp(\alpha), \alpha, \beta \in C^* , \beta \neq \alpha, \alpha^c \right\}.
\]
where $C^* :=C - \{\bf{0},\bf{1}\}$.  The \emph{code algebra} $A_C(\Lambda)$ is the commutative algebra over $\mathbb{F}$ with basis
\[
\{ t_i : i =1, \dots, n \} \cup \{ e^{\alpha} : \alpha \in C^* \},
\]
and multiplication given by
\begin{align*}
t_i \cdot t_j & = \delta_{i,j}t_i \\
t_i \cdot e^\alpha & = \begin{cases} 
a_{i,\alpha} \, e^\alpha & \text{if } \alpha_i = 1 \\
\mathrlap0\phantom{ \sum \limits_{i \in \supp(\alpha) }c_{i,\alpha} t_i} & \text{if } \alpha_i =0
\end{cases} \\
e^\alpha \cdot e^\beta & = \begin{cases}
b_{\alpha, \beta}\, e^{\alpha + \beta} & \text{if } \alpha \neq \beta, \beta^c \\
 \sum \limits_{i \in \supp(\alpha) }c_{i,\alpha} t_i & \text{if } \alpha = \beta  \\
0 & \text{if } \alpha = \beta^c
\end{cases}
\end{align*}
\end{definition*}

One particularly nice choice of structure parameters is where $a = a_{i, \alpha}$, $b = b_{\alpha, \beta}$ and $c= c_{i, \alpha}$ for all $i \in \supp(\alpha)$, $\alpha, \beta \in C^*$. We note that $A = A_C(\Lambda)$ embeds into the corresponding code VOA if $a = \frac{1}{4}$ and $c= 4b^2$.  We say the algebra $A$ is \emph{non-degenerate} if $\supp(C) = \{1, \dots, n\}$, $|C^*| >0$ and the structure parameters are all non-zero.

Code algebras are generically non-associative and they are almost always simple. Indeed, the following is our first important result.

\begin{theorem*}\label{intro:simple}
A non-degenerate code algebra $A_C$ is simple unless $C = \{ \bf{0}, \bf{1},\allowbreak \alpha, \alpha^c \}$.  In the latter case, the algebra has exactly two non-trivial proper ideals.
\end{theorem*}

A code algebra has some obvious idempotents, namely the $t_i$s. If $a_{i, \alpha} \neq 1$ for all $\alpha \in C^*$, then $t_i$ is primitive and semi-simple and we completely determine its fusion law which is given in Table \ref{Fusion law} of Proposition \ref{fusiontable}.  We note that, when the structure parameters are $(a,b,c)$ with $a \neq 1$ there are exactly three eigenvalues, $1$, $0$ and $a$, and moreover the fusion law is the same as for axial algebras of Jordan type \cite{Axial2}.  Furthermore, we show that $t_i$ satisfies the so-called \emph{Seress condition} (i.e.\ $A_0 A_\lambda \subseteq A_\lambda$, for all eigenvalues $\lambda \neq 1$) if and only if $t_i$ has exactly one eigenvalue not equal to $0$ or $1$.  (This is a small remainder of associativity in a non-associative setting, see \cite[Proposition 3.9]{Axial1}.)

If $a_{i, \alpha} \neq 1$ for all $\alpha \in C^*$ with $\alpha_i = 1$ and $\ch(\mathbb{F}) \neq 2$, the fusion law for $t_i$ induces a $\ZZ_2$-grading on the algebra, so we may define an involutory algebra automorphism $\tau_i$, called a \emph{Miyamoto involution}. Provided the structure parameters are regular (see Definition \ref{regular}), the automorphism group $\Aut(C)$ of the code also has a natural induced action on $A_C$.

\begin{theorem*}\label{intro:grp}
Suppose $\ch(\mathbb{F}) \neq 2$ and let $A_C$ be a non-degenerate code algebra with regular structure parameters where $a_{i, \alpha} \neq 1$ for all $i \in \supp(\alpha)$, $\alpha \in C^*$. Then,
\[
G = M{:}\Aut(C) \leq \Aut(A_C)
\]
where $M = \langle \tau_i : i = 1, \dots, n \rangle$.
\end{theorem*}

We may also define Frobenius forms on code algebras.

\begin{theorem*}\label{intro:frob}
A non-degenerate code algebra $A_C$ admits a Frobenius form if and only if conditions on the structure parameters {\textup(}see Theorem $\ref{frob}$\textup{)} are satisfied.  If so, the form is uniquely defined up to scaling by a choice of $\lambda_i$, $i = 1, \dots, n$, and is given by $(t_i,t_j)= \delta_{i,j} \lambda_i$, $(t_i, e^\alpha) = 0$ and $(e^\alpha, e^\beta) = \frac{c_{i, \alpha}}{a_{i, \alpha}}\lambda_i \delta_{\alpha,\beta}$, where $\alpha_i =1$, apart from one example with $C = \{\0, \1, \alpha, \alpha^c\}$.
\end{theorem*}

In comparison, Hall, Segev and Shpectorov show that axial algebras of Jordan type admit a unique Frobenius form \cite{Axial3}, but this is not known for general axial algebras.

In order for a code algebra to be an axial algebra, we must have enough idempotents to generate it.  Inspired by the example of the Hamming code VOA $V_{H_8}$, we make the following definition of the $s$-map. Given a constant weight subcode $D$ of $C$ (i.e.\ there is only one weight of codeword in $D^* := D -\{ \bf 0,1\}$) where the structure parameters supported on $D^*$ are constant $(a,b,c)$ and $v \in \mathbb{F}_2^n$, there exists an idempotent in $A_C$ of the form
\[
s(D,v) := \lambda \sum_{i \in \supp(D)} t_i + \mu \sum_{\alpha \in D^*} (-1)^{(v,\alpha)} e^\alpha
\]
where $\lambda, \mu \in \mathbb{F}$ satisfy a linear and quadratic equation respectively which are given in Proposition $\ref{smap}$. We will assume that we have taken the field large enough so that the quadratic equation has solutions.

All binary linear codes $C$ possess a constant weight subcode $D$. In fact, $D = \{ \bf 0, \alpha\}$ for $\alpha \in C^*$ is the smallest such example.  So, we call $s(D,v)$ a \emph{small idempotent}.  For idempotents in general it is difficult to find even the eigenvalues and eigenvectors, but for small idempotents we can do this.

\begin{theorem*}\label{intro:small}
Let $A_C$ be a non-degenerate code algebra on a constant weight code $C$ with structure parameters satisfying the conditions in Theorem $\ref{smallidem}$ and $a \neq \frac{1}{2|\alpha|}, \frac{1}{3|\alpha|}$. Then, the small idempotents are
\[
e_\pm := \lambda \sum_{i \in \supp(\alpha)} t_i \pm \mu e^\alpha
\]
where $\lambda = \frac{1}{2a|\alpha|}$ and $\mu^2 = \frac{\lambda - \lambda^2}{c}$. Furthermore, $e_\pm$ are primitive axes for the fusion law given in Table $\ref{tab:small}$.
\end{theorem*}

In some cases, the small idempotents generate the whole algebra and so we obtain the following.

\begin{corollary*}\label{intro:smallaxial}
Suppose $C$ is a simplex or first order Reed-Muller code and $A_C(a,b,c)$ is a non-degenerate code algebra with $a \neq \frac{1}{2|\alpha|}, \frac{1}{3|\alpha|}$.  Then $A$ is an axial algebra.
\end{corollary*}

The case of an arbitrary code (not constant weight) will be dealt with in an upcoming paper \cite{codealg2}.

The structure of the paper is as follows. In Section \ref{sec:background}, we give some basic results on linear codes and axial algebras. Code algebras are defined in Section \ref{sec:codealg} and we investigate their algebra structure, including their simplicity, automorphism group and Frobenius form, proving Theorems \ref{intro:simple}, \ref{intro:grp} and \ref{intro:frob}.  In Section \ref{sec:smap}, we give the $s$-map construction and investigate small idempotents, proving Theorem \ref{intro:small} and Corollary \ref{intro:smallaxial}. We present some examples in Section \ref{sec:examples}. The most important of these is the code algebra $A_{H_8}$ for the extended Hamming code $H_8$: this is an axial algebra that coincides with the degree 2 piece of the code VOA of $H_8$ \cite{M2,M3}. Finally, in Appendix \ref{sec:motivation}, we outline the construction of code VOAs which provides part of the motivation for the definition of code algebras.

\begin{notation}
Throughout the paper, we write statements involving $\1 \in C$, or the complement $\alpha^c$ of a codeword $\alpha$.  We do not assume that $\1 \in C$, or complements exist, just that if they do, then these statements should hold.
\end{notation}

\medskip

We would like to acknowledge a Heilbronn Collaboration Grant which made possible a visit to Bristol from the first and third authors and also a Mexican Academy of Sciences grant under the Newton Fund/CONACYT for a visit of the second author to Guadalajara.  We would like to thank Tim Burness for some helpful comments on a previous draft of this paper.

\section{Background}\label{sec:background}

We begin by reviewing some facts about codes and fixing notation, before giving the definition and some brief details about axial algebras.

\subsection{Binary linear codes}

Let $\mathbb{F}_2$ be the field with two elements. A \emph{binary linear code} $C$ of length $n$ and dimension $k$ is a $k$-dimensional subspace of $\mathbb{F}_2^n$.  We write $[n] := \{ 1, \dots, n \}$.  For any $\alpha = (\alpha_1, \dots, \alpha_n) \in \mathbb{F}_2^n$, denote its support by 
\[ \supp(\alpha) := \{ i \in [n]: \alpha_i = 1 \}, \]
and its Hamming weight by $\vert \alpha \vert :=  \vert \supp(\alpha) \vert$. The support of the code $C$ itself is defined to be $\supp(C) := \bigcup_{\alpha \in C} \supp(\alpha)$ and the weights of the codewords in $C$ is denoted $\mbox{wt}(C) := \{ \vert \alpha \vert : \alpha \in C \}$.

The \emph{\textup{(}Hamming\textup{)} distance between two codewords is $d(\alpha, \beta) = |\alpha-\beta|$.  The minimum distance of a code $C$ is the minimum distance between any two codewords.  For a linear code $C$, this is equal to the} minimum weight of a codeword in $C$.  A $[n,k,d]$-code is simply a $k$-dimensional binary linear code of length $n$ with minimum distance $d$.

Two codes $C$ and $D$ are \emph{similar} if there exists $g \in S_n$ such that $C^g = D$, where $S_n$ acts naturally on $C$ by permuting the coordinates of the codewords.  We define the automorphism group of $C$ as $\Aut(C) := \{ g \in S_n : C^g = C \}$.

We write $C^*$ for the non-constant codewords in $C$; that is, all codewords which are not $\0 := (0, \dots, 0)$, or $\1 := (1, \dots, 1)$. If ${\bf 1 } \in C$, then every $\alpha \in C$ has a complement, denoted by $\alpha^c :=  \1 + \alpha$.  Conversely, if some $\alpha \in C$ has a complement, then $\1 \in C$ and every codeword in $C$ has a complement.

Consider the usual dot product $(\cdot ,\cdot ) : \mathbb{F}_2^n \times \mathbb{F}_2^n \to \mathbb{F}_2$ given by $(u, v) := \sum_{i=1}^n u_i v_i$ for $u,v \in \mathbb{F}_2^n$.  For any $v \in \mathbb{F}_2^n$ and $k \in \{ 0,1\}$, we define
\[
C_k(v) := \{ \alpha \in C : (\alpha, v) = k \}.
\]
When the $v$ is clear, we write $C_i$ instead of $C_i(v)$.  We define $C_i^* := C_i - \{\0, \1\}$.  Note that by definition we have $\alpha \in C_{(\alpha,v)}$.

\begin{lemma}\label{innerC}
Let $C \subseteq \mathbb{F}_2^n$ be a binary linear code of length $n$. Let $v \in \mathbb{F}_2^n$ and $C_k = C_k(v)$, for $k \in \{0,1 \}$. Then:
\begin{enumerate}
\item[$1.$] The map $\alpha \mapsto (\alpha,v)$ is a homomorphism from $C$ to $\mathbb{F}_2$, viewed as additive groups. Hence, $C$ is the disjoint union of $C_0$ and $C_1$.
\item[$2.$] $C_0$ is a binary linear code, so, in particular, it is non-empty.
\item[$3.$] If $C_1$ is non-empty, then $|C_0| = |C_1| = \frac{|C|}{2}$.
\end{enumerate}
Furthermore, if ${\bf 1 } \in C$, then:
\begin{enumerate}
\item[$4.$] If $v$ has odd weight, then the complements of codewords in $C_0$ lie in $C_1$ and vice versa.
\item[$5.$] If $v$ has even weight, then each $C_k$ is closed under taking complements.
\end{enumerate}
\end{lemma}
\begin{proof}
It is easy to see that the map given in part 1 is a homomorphism and the remainder of the Lemma follows from this.
\end{proof}

\subsection{Axial algebras}

In this section, we will review the basic definitions related to axial algebras. For further details, see \cite{Axial1,Axial2}. Let $\mathbb{F}$ be a field not of characteristic two, $\mathcal{F} \subseteq \mathbb{F}$ a subset, and $\star : \mathcal{F} \times \mathcal{F} \to 2^{\mathcal{F}}$ a symmetric binary operation. We call the pair $(\mathcal{F}, \star)$ a \emph{fusion law over $\mathbb{F}$}. The fusion law is \emph{$G$-graded}, where $G$ is a finite abelian group, if there exist a partition $\{ \mathcal{F}_g\}_{g\in G}$ of $\mathcal{F}$ such that $a \star b \subseteq \mathcal{F}_{gh}$ for all $a \in \mathcal{F}_g$, $b \in \mathcal{F}_h$, $g,h \in G$. 

Let $A$ be a non-associative (i.e.\ not-necessarily-associative) commutative algebra over $\mathbb{F}$. For an element $a \in A$, the adjoint endomorphism $\ad_a$ is defined by $\ad_a(v):=av$, $\forall v \in A$. Let $\Spec(a)$ be the set of eigenvalues of $\ad_a$, and for $\lambda \in \Spec(a)$, let $A_\lambda(a)$ be the $\lambda$-eigenspace of $\ad_a$. Where the context is clear, we will write $A_\lambda$ for $A_\lambda(a)$.

\begin{definition}
Let $(\mathcal{F}, \star)$ be a fusion law over $\mathbb{F}$. An element $a \in A$ is an \emph{$\mathcal{F}$-axis} if the following hold:
\begin{enumerate}
\item $a$ is \emph{idempotent} (i.e.\ $a^2 = a$),
\item $a$ is \emph{semisimple} (i.e.\ the adjoint $\ad_a$ is diagonalisable),
\item $\Spec(a) \subseteq \mathcal{F}$ and $A_\lambda A_\mu \subseteq \bigoplus_{\gamma \in \lambda \star \mu } A_{\gamma}$, for all $\lambda, \mu \in \Spec(a)$. 
\end{enumerate}
We say that an $\mathcal{F}$-axis $a$ is \emph{primitive} if $A_1 = \langle a \rangle$.
\end{definition}

\begin{definition}
An \emph{axial algebra} is a pair $(A, X)$, where $A$ is a non-associative commutative algebra and $X$ is a set of $\mathcal{F}$-axes that generate $A$.  We say the axial algebra is \emph{primitive} if all the axes in $X$ are primitive.
\end{definition}

When the fusion law is clear from context we drop the $\mathcal{F}$ and simply use the term axis and axial algebra. We will also abuse notation and just write $A$ for an axial algebra $(A, X)$.

\begin{definition}\label{FrobeniusAxial}
Let $A$ be an $\mathcal{F}$-axial algebra.  A \emph{Frobenius form} is a non-zero bilinear form $( \cdot , \cdot ) : A \times A \to \mathbb{F}$ that associates. That is, for all $x,y,z \in A$,
\[
(x,yz)=(xy,z)
\]
\end{definition}

Sometimes in the literature it is also required that $(a,a) = 1$ for each $a \in X$, however we will not require this. In the context of VOAs, the value $\frac{1}{2} (a,a)$, where $a$ is an $\mathcal{F}$-axis, is called the \emph{central charge} of $A$.

In particular, a \emph{Majorana algebra} is an axial algebra over $\mathbb{R}$ with a positive definite Frobenius form, where $\mathcal{F}= \{ 0,1, \frac{1}{4},\frac{1}{32} \}$ with fusion law $\star$ given by \cite[Table 1]{IPSS10}. These kinds of algebra generalise subalgebras of the Griess algebra.


\section{Code algebras}\label{sec:codealg}

Inspired by code VOAs and Theorem \ref{motivatingthm}, we will now introduce our main definition which is the subject of this paper.

Let $C \subseteq \mathbb{F}_2^n$ be a binary linear code of length $n$ and $\mathbb{F}$ a field. Recall from Definition \ref{CodeAlgebra} that a collection of \emph{structure parameters} is a subset of $\mathbb{F}$
\[
\Lambda := \left\{ a_{i,\alpha}, b_{\alpha,\beta}, c_{i,\alpha} \in k  : i \in \supp(\alpha), \alpha, \beta \in C^* , \beta \neq \alpha, \alpha^c \right\}.
\]
The code algebra $A_C(\Lambda)$ is the commutative algebra over $\mathbb{F}$ with basis
\[
\{ t_i : i \in [n] \} \cup \{ e^{\alpha} : \alpha \in C^* \},
\]
and multiplication given by in Definition \ref{CodeAlgebra}.

The code algebra $A_C(\Lambda)$ has dimension $n + \vert C^* \vert$. Note that the set of structure parameters defined above gives some of the structure constants for the algebra, while the remaining structure constants are all zero. One particularly nice choice of structure parameters are $\Lambda = \{a_{i,\alpha} = a, b_{\alpha, \beta} = b, c_{i, \alpha} = c\}$ which we will write $(a,b,c)$. The basis elements $t_i$ and $e^\alpha$ are called \emph{toral} and \emph{codeword elements}, respectively. It will sometimes be convenient to abuse notation by writing $t_\alpha = \sum_{i \in \supp(\alpha)} t_i$ where $\alpha \in C^*$.

\begin{remark}
Since the algebra is commutative by definition, we must have $b_{\alpha, \beta} = b_{\beta, \alpha}$ for all $\alpha, \beta \in C^*$. However, $A_C(\Lambda)$ is non-associative in general.  In fact, if $A_C(\Lambda)$ is associative, then all the $a$ and $c$ structure parameters must be zero.
\end{remark}

As the following example shows, the algebra $A_C(\Lambda)$ is not in general even power associative.

\begin{example}
Suppose $A_C(a,b,c)$ is a code algebra and $x = e^\alpha$ for some $\alpha \in C^*$.  Then $x^2 = c t_\alpha$
\begin{align*}
(x^2)^2 &= c t_\alpha \\
x(x \cdot x^2) &= a c^2 |\alpha| t_\alpha
\end{align*}
which is not equal in general and hence the algebra is not power associative.
\end{example}

%
%

We want to impose a non-degeneracy condition on code algebras.

\begin{definition}
A code algebra $A_C(\Lambda)$ is \emph{non-degenerate} if $\supp(C) = [n]$, $|C^*| > 0$ and none of the structure parameters in $\Lambda$ are zero.
\end{definition}

\begin{example}[Code VOA example]
The algebra described in Theorem \ref{motivatingthm} is a non-degenerate code VOA with structure parameters $(a, b, c) = (\frac{1}{4}, \lambda, 4 \lambda^2)$. We will call such a choice of structure parameters \emph{code VOA structure parameters}.
\end{example}

Some algebras with specific choices of structure parameters will be more interesting than others.  In particular, we will be interested in algebras with a large automorphism group and this will impose restrictions on the structure parameters.  One obvious place the automorphisms may come from is from the code itself. We first need a definition.

\begin{definition}\label{regular}
Let $G \leq \Aut(C)$. The structure parameters $\Lambda$ are called \emph{$G$-regular} if for all $g \in G$
\begin{enumerate}
\item[$1.$] $a_{i,\alpha} = a_{ig, \alpha g}$ for all $i \in \supp(\alpha)$, $\alpha \in C^*$
\item[$2.$] $b_{\alpha, \beta} = b_{\alpha g, \beta g}$ for all, $\alpha, \beta \in C^*$, $\beta \neq \alpha, \alpha^c$
\item[$3.$] $c_{i, \alpha} = c_{i g, \alpha g}$ for all $i \in \supp(\alpha)$, $\alpha \in C^*$
\end{enumerate}
They are \emph{regular} if they are $\Aut(C)$-regular.
\end{definition}

Note that $\Lambda$ being $G$-regular just means that $b_{\alpha, \beta}$ is constant on $G$-orbits of $C \times C$ and $a_{i, \alpha}$ and $c_{i, \alpha}$ are both constant on $G$-orbits of $\mathbb{F}_2^n \times C$.

Let $G \leq  \Aut(C)$ and consider the mapping $\phi : G \to \Aut(A_C)$, where for each $g \in G$, $\phi(g)$ is the linear extension to $A_C$ of
\begin{align*}
 t_i^g & := t_{i g^{-1}}\\
(e^\alpha)^g &:= e^{\alpha g^{-1}}
\end{align*}

\begin{lemma}\label{inducedaction}
Let $G \leq \Aut(C)$.  The above mapping $\phi : G \to \Aut(A_C)$ is a well-defined group homomorphism \textup{(}i.e. an algebra representation of $G$ on $A_C$\textup{)} if and only if the structure parameters are $G$-regular. Moreover, when this representation is defined it is faithful.
\end{lemma}
\begin{proof}
Since $G$ has a well-defined action on the codewords of $C$, and so $C^*$, and also on $\mathbb{F}^n$, it is clear that $\phi$ is a group homomorphism. It remains to check whether, for all $g \in G$, $\phi(g)$ respects multiplication in the algebra.  Checking this, we find that the above are necessary and sufficient conditions.  When the action is defined it is clear that it is faithful as $\Aut(C)$ acts faithfully on $\mathbb{F}^n$.
\end{proof}

We note that the code VOA example satisfies these conditions.

\subsection{Subalgebras and simplicity}

A subalgebra of $A_C(\Lambda)$ which has a basis of idempotents that pairwise multiply to $0$ is called a \emph{torus}; clearly, this is always associative. The subalgebra $\langle t_i : i \in [n] \rangle$ is an example of such a torus; we call it the \emph{standard torus}.  It is easy to see that it is maximal. Indeed, suppose there exists a non-zero idempotent $x \in A_C$ which could be added to the standard torus.  If it is supported on any $t_i$, or any $e^\alpha$ where $\alpha_i =1$, then $t_i x \neq 0$, a contradiction.  However, this includes all basis elements, so $x =0$ and the standard torus is maximal. In the context of Majorana algebras, maximal tori have been studied and classified for low-dimensional cases \cite{CR15}.

It is easy to see that there are some other subalgebras of $A_C$ which are induced from subcodes of $C$.

\begin{lemma}\label{subalgebra}
Let $A_C$ be an arbitrary code algebra and $D$ a subcode of $C$.  Then $D$ defines a subalgebra
\[
\langle t_i, e^\alpha : \alpha \in D, i \in \supp( D) \rangle.
\]
\end{lemma}
\begin{proof}
We just need to check that the multiplication of the generators of the subalgebra is closed.  Any toral element multiplied by a codeword element is either zero or is a multiple of the same codeword element.  Since $D$ is a subcode, multiplication of two distinct codeword elements is closed.  Finally, since we include all the toral support for each codeword, multiplication of a codeword element by itself is also closed.
\end{proof}

Recall that an algebra $A$ is simple if it has no non-trivial proper ideals.

\begin{theorem}\label{simple}
A non-degenerate code algebra $A_C$ is simple unless $C = \{ \bf 0,1,\alpha, \alpha^c \}$.
\end{theorem}
\begin{proof}
Let $I$ be a non-trivial ideal of $A$.  Let $x \in I$ such that $x \neq 0$. We write $x = \sum_{i \in [n]} \lambda_i t_i + \sum_{\alpha \in C^*} \lambda_\alpha e^\alpha$. There are two cases.

Suppose $\lambda_\alpha \neq 0$ for some $\alpha \in C^*$.  We may choose $\alpha$ so that there does not exist $\beta \in C^*$ with $\supp(\beta) \supsetneqq \supp(\alpha)$ and $\lambda_\beta \neq 0$.  Then,
\[
( t_{j_1} ( \dots (t_{j_k} x )) \dots ) =  a_{j_1,\alpha} \dots a_{j_k,\alpha} \lambda_\alpha e^\alpha \in I
\]
where $\supp(\alpha) = \{ j_1, \dots, j_k \}$.  Hence $e^\alpha \in I$.  Observe that, for any $\beta \in C - \{ { \bf 0,1,} \alpha, \alpha^c \}$, we have $e^\beta \in I$ because $e^{\alpha + \beta} e^\alpha = b_{\alpha + \beta, \beta} e^\beta \in I$. Furthermore, as $\alpha^c = \beta + (\beta + \alpha^c)$, where $\beta, \beta + \alpha^c \in C - \{ 0,1,\alpha, \alpha^c \}$, we also have $e^{\alpha^c} = \frac{1}{b_{\beta + \alpha^c,\beta}} e^{\beta+\alpha^c} e^{\beta} \in I$.  So, $e^\beta \in I$ for all $\beta \in C^*$. Observe that $(e^\beta)^2 = \sum_{i \in \supp(\beta)} c_{i, \beta} t_i$ for any $\beta \in  C^*$. Hence, for all $j \in [n]$ we may find $\beta \in C^*$ such that $j \in \supp(\beta)$, so $t_j = \frac{1}{c_{i, \beta}} t_j (e^\beta)^2 \in I$. Therefore, $I = A$.

Now suppose $\lambda_\alpha = 0$ for all $\alpha \in C^*$, then $\lambda_i \neq 0$ for some $i \in [n]$. It is easy to see that $t_i = \frac{1}{\lambda_i} t_i x \in I$. Now, $e^{\beta} = \frac{1}{a_{i, \alpha}} t_i e^\beta \in I$ for any $\beta \in C^*$ with $i \in \supp(\beta)$. Therefore, $x' = e^{\beta}$ is as in the case above and we conclude that $I = A$.
\end{proof}

We now see that the condition on the above proposition was indeed necessary.

\begin{lemma}\label{non-simpleA}
If $C = \{ {\bf 0,1,}\alpha, \alpha^c \}$, then a non-degenerate code algebra $A_C$ has exactly two non-trivial proper ideals
\[
\langle \sum_{i \in \supp(\alpha)} t_i , e^\alpha \rangle, \qquad \langle \sum_{i \in \supp(\alpha^c)} t_i , e^{\alpha^c} \rangle.
\]
\end{lemma}
\begin{proof}
This is an easy calculation.
\end{proof}

\subsection{Idempotents}

Code algebras are constructed in such a way that there are some obvious idempotents.  Throughout this section we assume that $A= A_C(\lambda)$ is non-degenerate.

Recall from Section \ref{sec:background} the adjoint transformation and the notation
\[
A_\mu(v) := \{ w \in A : \ad_v (w) = \mu w \}.
\]
Where it is clear from the context, we will write $A_\mu$ for $A_\mu(v)$. 

\begin{lemma}\label{ideigenspaces}
For any $i \in [n]$, $\ad_{t_i}$ is semisimple with eigenvalues $1$, $0$ and the set $\{ a_{i, \alpha} : i \in \supp(\alpha), \alpha \in C ^* \}$ and if $a_{i,\alpha} \neq 1$ for all $\alpha \in C^*$ with $\alpha_i =1$ then
\[
A = A_1 \oplus A_0 \oplus \bigoplus A_{a_{i, \alpha}},
\]
where $A_1 = \langle t_i \rangle$, $A_0 = \langle t_j, e^\alpha : j \in [n], j \neq i, \alpha_i = 0 \rangle$ and $A_{a_{i, \alpha}} = \langle e^\beta : \beta_i =1, a_{j, \beta}=a_{i, \alpha} \rangle$.  Moreover, $A_0$ has dimension $\frac{|C|}{2} + n-2$.
\end{lemma}
\begin{proof}
The subspaces are clear from the definition of code algebras.  By Lemma \ref{innerC}, we see that the number of $\alpha \in C^*$ such that $\alpha_i=0$ is $\frac{|C|}{2}-1$.
\end{proof}

We note that the $1$-eigenspace is spanned by $t_i$.  That is, $t_i$ is a primitive idempotent. We can now give the fusion law for $t_i$.  In the table, rows and columns correspond to eigenspaces and the entries are the sum of the eigenspaces in which the product of two elements in the corresponding eigenspaces lies.  We label the eigenspaces by their eigenvalues and adopt the convention that an empty entry represents a zero product.

\begin{proposition}\label{fusiontable}
Suppose that the eigenvalues of $t_i$ are $\{1, 0, a_1, \dots, a_k\}$ and that if $a_{i,\alpha} \neq 1$ for all $\alpha \in C^*$ with $\alpha_i =1$.  Then the fusion law with respect to $t_i$ is given by Table $\ref{Fusion law}$.
\begin{table}[!htb]
\setlength{\tabcolsep}{4pt}
\renewcommand{\arraystretch}{1.5}
\centering
\begin{tabular}{c||c|c|c|c|c}
 & $1$ & $0$ & $a_1$ & $\dots$ & $a_k$ \\ \hline \hline
$1$ & $1$ &  & $a_1$ & $\dots$ & $a_k$ \\ \hline
$0$ &  & $0$ & $a_1, \dots ,a_k$ & $\dots$ & $a_1, \dots ,a_k$ \\ \hline
$a_1$ & $a_1$ & $a_1, \dots ,a_k$ & $1, 0$ & $0$ & $0$ \\ \hline
$\vdots$ & $\vdots$ & $\vdots$ & $0$ & $\ddots$ & $0$ \\ \hline
$a_k$ & $a_k$  & $a_1, \dots, a_k$ & $0$ & $0$ & $1, 0$ 
\end{tabular}
\caption{Fusion law with respect to $t_i$.} \label{Fusion law}
\end{table}
\end{proposition}
\begin{proof}
The first row and column is clear since $t_i$ is a primitive idempotent.  For the row corresponding to $A_0 = \langle t_j, e^\alpha : j \in [n], j \neq i, \alpha_i = 0 \rangle$, observe that the product of any two toral elements from here gives either the same toral element or zero.  The product of two codeword elements gives another with a $0$ in the $i$\textsuperscript{th} position and the product of a toral and codeword element gives a multiple of the same codeword element.  Any codeword element $e^\alpha$ in $A_{a_j}$ has $\alpha_i = 1$, so its product with a codeword element in $A_0$ will also have a $1$ in the $i$\textsuperscript{th} position.  Hence, it will be a non-zero eigenvector for $t_i$, but not necessarily for $a_j$.  The product of $e^\alpha$ with a toral element in $A_0$ will either be a scalar multiple of $e^\alpha$, or zero.

Finally, let $e^\alpha \in A_{a_j}$ and  $e^\beta \in A_{a_k}$.  If $j \neq k$, then $e^\alpha e^\beta$ is either $b_{\alpha,\beta} e^{\alpha + \beta}$ or $0$.  Since $\alpha_i = \beta_i = 1$, the $i$\textsuperscript{th} position of $\alpha+\beta$ is $0$, hence the product is in the zeroth eigenspace.  If $j=k$, then either $\alpha \neq \beta$ and the product is in the zeroth eigenspace as above, or $\alpha = \beta$ and the product is the sum of the toral elements in $\supp(\alpha)$, one of which is $t_i$.
\end{proof}

Note that for a given code and structure parameters, the product of the two eigenspaces may actually be a smaller set than that which is suggested by the fusion law in Table \ref{Fusion law}. Also, if some $a_{i, \alpha}=1$, then this just means that two of the columns (and the corresponding rows) are merged in the fusion table.  This difficulty can be overcome by defining the fusion law to be defined on a set of indeterminates $F$ rather than the eigenvalues and then having a map from $F$ to the eigenvalues.

\begin{corollary}\label{t_iaxis}
If $a_{i, \alpha} \neq 1$ for all $\alpha \in C^*$, then $t_i$ is a primitive $\mathcal{F}$-axis, where $\mathcal{F}$ is the fusion law given in Table $\ref{Fusion law}$.
\end{corollary}

\begin{example}[Code VOA example]
The structure parameters in the code VOA example are $(a,b,c) = (\frac{1}{4}, \lambda, 4\lambda^2)$, so we have the fusion law given in Table \ref{tab:Majfusion}.

\begin{table}[!htb]
\setlength{\tabcolsep}{4pt}
\renewcommand{\arraystretch}{1.5}
\centering
\begin{tabular}{c||c|c|c}
 & $1$ & $0$ & $\frac{1}{4}$ \\ \hline \hline
$1$ & $1$ &  & $\frac{1}{4}$ \\ \hline
$0$ &  & $0$ & $\frac{1}{4}$ \\ \hline
$\frac{1}{4}$ & $\frac{1}{4}$  & $\frac{1}{4}$ & $1, 0$ 
\end{tabular}
\caption{VOA fusion law with respect to $t_i$.}\label{tab:Majfusion}
\end{table}
\end{example}

Suppose that $a_{i,\alpha} \neq 1$ for all $\alpha \in C^*$ with $\alpha_i =1$.  Then, we see from Table \ref{Fusion law} that the fusion law for $t_i$ is $\mathbb{Z}_2$-graded.  So, this induces a $\ZZ_2$-grading on $A$:
\[
A = A_+ \oplus A_-,
\]
where $A_+ := A_0 \oplus A_1$ and $A_- := \bigoplus A_{a_{i, \alpha}}$. If $\ch(\mathbb{F}) \neq 2$, there is a natural algebra automorphism of order at most two, denoted by $\tau_i \in \Aut(A)$, which acts trivially on $A_+$ and negates the vectors in $A_-$. We call $\tau_i$ the \emph{Miyamoto involution} associated to $t_i$.  Note that when $A$ is non-degenerate, $A_-$ is non-empty and $\tau_i$ does indeed have order two.

We call the group $M = \langle \tau_i : i = 1, \dots, n\rangle$ generated by the Miyamoto involutions, the \emph{Miyamoto group}.

\begin{proposition}
Suppose $\ch(\mathbb{F}) \neq 2$ and let $A_C$ be a non-degenerate code algebra with regular structure parameters where $a_{i, \alpha} \neq 1$ for all $i \in \supp(\alpha)$, $\alpha \in C^*$.  Define $G \leq \Aut(A)$ to be the group generated by $\Aut(C)$ and the Miyamoto group $M$.  Then 
\[
G = M{:}\Aut(C)
\]
is a semi-direct product of $M = \langle \tau_i : i = 1, \dots, n\rangle$ by $\Aut(C)$.
\end{proposition}
\begin{proof}
By Lemma \ref{inducedaction}, $\Aut(C)$ has a well-defined faithful action on $A_C$. In particular, it acts faithfully on the $t_i$ by permuting the indices.  As $M$ fixes all the indices, $M \cap \Aut(C) = 1$.  Let $\alpha \in C$ and $i = 1, \dots, n$.  Since $\alpha_i = (\alpha g)_{ig}$ for all $g \in \Aut(C)$, we see that
\[
\tau_i g = g \tau_{ig}
\]
In particular, we see that $M \Aut(C) = \Aut(C)M$  is a group and hence $G = M\Aut(C)$.  Moreover, $M$ is a normal subgroup of $G$ and so we have a semi-direct product.
\end{proof}

The fusion law for an idempotent $e$ satisfies the \emph{Seress condition} if $1,0 \in \mathcal{F}$ and
\[
A_0 A_\lambda \subseteq A_\lambda
\]
for all eigenvalues $\lambda \neq 1$ (see \cite[Section 2.2]{Axial2}).

\begin{proposition}\label{Seress}
Suppose that $a_{i, \alpha} \neq 1$ for all $i \in \supp(\alpha)$, $\alpha \in C^*$.  Then, the fusion law for $t_i$ satisfies the Seress condition if and only if $t_i$ has at most one eigenvalue which is not $0$ or $1$.
\end{proposition}
\begin{proof}
By Lemma \ref{fusiontable}, it is clear that the fusion table satisfies the Seress condition if $t_i$ has at most one eigenvalue which is not $0$ or $1$.  Conversely, suppose that the fusion law for $t_i$ satisfies the Seress condition and there is at least one eigenvalue $\lambda \neq 0,1$.  By Lemma \ref{ideigenspaces}, all $e^\alpha$ such that $\alpha_i=0$ are in the $0$-eigenspace.  Let $e^\beta \in A_\lambda$.  By assumption, the fusion law satisfies the Seress condition and hence $b_{\alpha, \beta}e^{\alpha+\beta} = e^\beta e^\alpha \in A_\lambda$ for all $\alpha \in C_0(i)$ such that $\alpha \neq \beta^c$.  By Lemma \ref{innerC}, this has dimension at least $\frac{|C|}{2} -1$ if $\1 \in C$ and $\frac{|C|}{2}$ if $\1 \notin C$.  By Lemma \ref{ideigenspaces}, $\dim A_1 \oplus A_0 = n + \frac{|C|}{2} - 1$.  Since $\dim A$ is $n + |C|-2$ if ${\bf 1 } \in C$ and $n + |C|-1$ if $\1 \notin C$, by a counting argument we see that $\lambda$ can be the only non-trivial eigenvalue.
\end{proof}

We note that if there is exactly one eigenvalue $a$ of $t_i$ which is not equal to $1$ or $0$, then the fusion law is the same as for axial algebras of Jordan type $a$ (see \cite{Axial2}).  The following will be allow us to show when the algebra is unital.

\begin{lemma}\label{tscaling}
Suppose $A_C$ is a non-degenerate code algebra.  Then, there exists an element $t \in A_C(\lambda)$ such that $t \cdot t_i = t_i$ and $t \cdot e^\alpha$ is a scalar multiple of $e^\alpha$ for all $i \in [n]$, $\alpha \in C^*$.  Moreover, there is a unique such element, it is an idempotent and it is given by $t = \sum\limits_{i=1}^n t_i$.
\end{lemma}
\begin{proof}
We write $t = \sum \limits_{i=1}^n \lambda_i t_i + \sum \limits_{\alpha \in C^*} \lambda_\alpha e^\alpha$, for some $\lambda_i, \lambda_\alpha \in \mathbb{R}$.   Suppose that $\lambda_\beta \neq 0$ for some $\beta \in C^*$.  We calculate products:
\[
te^\beta = \sum \limits_{i \in \supp(\beta)} \lambda_i a_{i, \beta}\, e^\beta + \sum \limits_{\alpha \neq \beta, \beta^c} \lambda_\alpha b_{\alpha, \beta}\, e^{\alpha + \beta} + \lambda_\beta \sum \limits_{i \in \supp(\beta)} c_{i, \beta} t_i
\]
By assumption this is some scalar multiple of $e^\beta$.  So, by considering coefficients of $t_i$, we see that $c_{i, \beta} = 0$ for all $i \in [n]$, contradicting $A_C$ being non-degenerate. Hence, $\lambda_\beta = 0$ for all $\beta \in C^*$.

Considering the product with $t_j$, we now see that
\[
t t_j = \sum \limits_{i=1}^n \lambda_i t_i t_j = \lambda_j t_j
\]
and hence we see that $\lambda_j = 1$ for all $j \in [n]$.  Therefore, when such a $t$ does exist it is $\sum\limits_{i=1}^n t_i$ and so is unique.  It is easy to see that such a $t$ is an idempotent.
\end{proof}


\begin{corollary}
Let $A_C$ be a non-degenerate code algebra such that for every $\alpha \in C^*$, $a_{i, \alpha} = a_{j, \alpha}$ for all $i,j \in \supp(\alpha)$.  Then $A$ has an identity if and only if $a_{i, \alpha} = \frac{1}{|\alpha|}$ for all $i \in \supp(\alpha), \alpha \in C^*$.
\end{corollary}
\begin{proof}
By Lemma \ref{tscaling}, the only candidate for an identity is $t = \sum t_i$.  Since for each $\alpha \in C^*$, we have
\[
t e^\alpha = \sum_{i \in \supp(\alpha)} a_{i, \alpha} e^\alpha
\]
it is clear that $t$ is the identity if and only if $a_{i, \alpha} = \frac{1}{|\alpha|}$ for all $i \in \supp(\alpha)$.
\end{proof}


\subsection{Frobenius form}

We now made the analogous definition for code algebras as for axial algebras (see Definition \ref{FrobeniusAxial}).

\begin{definition}
A \emph{Frobenius form} on a non-degenerate code algebra $A$ is a bilinear form $(\cdot, \cdot) : A \times A \to \mathbb{F}$ such that the form associates.  That is, for all $x,y,z \in A$,
\[
(x,yz) = (xy,z)
\]
\end{definition}

We now collect some basic facts about Frobenius forms. The following is adapted from \cite[Proposition 3.5]{Axial1}.

\begin{lemma}
An associative bilinear form on a non-degenerate code algebra $A$ is symmetric.
\end{lemma}
\begin{proof}
We need only show the result on a basis of $A$; we use the standard basis consisting of $t_1, \dots, t_n$ and $e^\alpha$, $\alpha \in C^*$.  We note that for each $b$ in the basis, $b = b_1b_2$ where $b_1, b_2$ are some scalar multiple of a basis element.  Explicitly, $t_i = t_i t_i$ and $e^\alpha = (\frac{1}{a_{i,\alpha}} t_i) e^\alpha$ where $\alpha_i = 1$.  Now for some other element $c \in A$ we have
\[
(b,c) = (b_1b_2,c) = (b_1,b_2c) = (b_1,cb_2) = (b_1c,b_2) = (cb_1,b_2) = (c,b).
\]
\end{proof}

\begin{lemma}{\normalfont \cite{Axial1}}\label{Frobperp}
Suppose $A$ is a non-degenerate code algebra which admits a Frobenius form.  Then, the eigenspaces of a semisimple element of $A$ are perpendicular.
\end{lemma}
\begin{proof}
Let $x \in A^\lambda_a$, $y \in A^\mu_a$, $\lambda \neq \mu$, with respect to a semisimple element $a$.  We have
\[
\lambda(x,y) = (xa,y) = (x,ay) = \mu (x,y)
\]
Since $\lambda \neq \mu$, $(x,y) =0$.
\end{proof}

In particular, we may take the semisimple element to be an $\mathcal{F}$-axis in the above lemma.

\begin{theorem}\label{frob}
Let $A_C(\Lambda)$ be a non-degenerate code algebra. Then, $A_C(\Lambda)$ admits a Frobenius form if and only if there exist constants $\lambda_1, \dots, \lambda_n \in \mathbb{F}$ such that
\begin{enumerate}
\item[$1.$] $\lambda_\alpha := \frac{c_{i,\alpha}}{a_{i, \alpha}} \lambda_i$ is constant for all $i \in \supp(\alpha)$.
\item[$2.$] $b_{\alpha, \beta} \lambda_\gamma = b_{\alpha, \gamma} \lambda_\beta = b_{\beta,\gamma} \lambda_\alpha$ for all $\alpha, \beta, \gamma \in C^*$ such that $\alpha + \beta = \gamma$.
\end{enumerate}
Furthermore, when the Frobenius form does exist, it is given by
\begin{align*}
(t_i, t_j ) & = \delta_{i,j} \lambda_i \\
(t_i , e^\alpha) &= \begin{cases}
\lambda_{i, \alpha} & \mbox{if } C = \{\0, \1, \alpha, \alpha^c\}, |\alpha| = 1, a_{\supp(\alpha), \alpha}=1\\
0 & \mbox{otherwise} \end{cases} \\
(e^\alpha, e^\beta) & = \lambda_\alpha \delta_{\alpha, \beta}
\end{align*}
In the exceptional case, where $C = \{\0, \1, \alpha, \alpha^c\}$ with $|\alpha| = 1$ and $a_{\supp(\alpha), \alpha}=1$, the Frobenius form is uniquely defined \textup{(}up to scaling\textup{)} by the $\lambda_i$ and $\lambda_{i, \alpha}$ and otherwise it is uniquely defined \textup{(}up to scaling\textup{)} by the $\lambda_i$.
\end{theorem}
\begin{proof}
Consider a potential such form and define $\lambda_i = (t_i, t_i)$ for all $i \in [n]$. By Lemma \ref{Frobperp}, the eigenspaces of a given $\mathcal{F}$-axis are perpendicular.  So, $(t_i, t_j) = 0$ for $i \neq j$ and, if $a_{i, \alpha} \neq 1$, $(t_i, e^\alpha) =0$ also.  Now, if $|\alpha| \geq 2$, then there exists $j \in \supp(\alpha)$ such that $j \neq i$.  Since $C$ is non-degenerate, $e^\alpha$ and $t_i$ are in different eigenspaces with respect to $t_j$ and so $(t_i, e^\alpha)=0$.

If $\alpha, \beta \in C^*$ are distinct, then there is some position $i$ such that $i \in \supp(\alpha)$ and $i \not \in \supp(\beta)$. So, $e^\alpha$ and $e^\beta$ lie in different eigenspaces with respect to $t_i$ and hence $(e^\alpha, e^\beta)=0$.  Therefore, any Frobenius form is diagonal with respect to our chosen basis, apart from possibly if $|\alpha|=1$ and $a_{i, \alpha}=1$.  We write $\lambda_\alpha := (e^\alpha, e^\alpha)$.

Observe that
\[
(t_i \cdot e^\alpha, e^\beta )  = \begin{cases}
a_{i, \alpha} \lambda_\alpha & \text{if } \alpha = \beta, \alpha_i = 1, \\
0 & \text{otherwise}.
\end{cases}
\]
Whereas,
\[
(t_i, e^\alpha \cdot e^\beta )  = \begin{cases}
c_{i,\alpha}\lambda_i & \text{if } \alpha = \beta, \alpha_i = 1, \\
0 & \text{otherwise}.
\end{cases}
\]
and so we see that we must have $\lambda_\alpha = \frac{c_{i,\alpha}}{a_{i,\alpha}}\lambda_i$, for $(\cdot, \cdot)$ to be a Frobenius form.  This provides the first of our conditions.

We have
\[
(e^\alpha, e^\beta \cdot  e^\gamma ) = \begin{cases}
b_{\beta, \gamma} \lambda_\alpha & \text{if } \alpha = \beta + \gamma, \\
\sum_{i \in \supp(\beta)} c_{i, \beta} (e^\alpha, t_i) & \mbox{if }\beta = \gamma, \\

0 & \text{otherwise}.
\end{cases}
\]

First we consider the case where $|\alpha|=1$.  Let $\{i\} = \supp(\alpha)$.  We claim that either there exists $\beta \in C_i^*$ such that $\beta \neq \alpha$, or $C = \{\0, \1, \alpha, \alpha^c\}$.  Suppose we don't have such a $\beta$; then, $C_i^*= \{\alpha\}$. Now, either $1 \notin C$ and so $C_i = C_i^*= \{\alpha\}$ which implies $C = \{0, \alpha = \1\}$, a contradiction.  Or, $1 \in C$, $C_i = \{  \alpha, \1\}$ and so $C= \{\0, \1, \alpha, \alpha^c\}$, proving the claim.

We consider the above with $\beta = \gamma \neq \alpha$.  If $C = \{\0, \1, \alpha, \alpha^c\}$, then $\beta = \alpha^c$ and we have $(e^\alpha e^{\alpha^c}, e^{\alpha^c})0 = (e^\alpha, e^{\alpha^c} e^{\alpha^c})$.  Otherwise, there exists $\beta \in C_i^*$ such that $\alpha \neq \beta$.  Now, for the form to be Frobenius, we must have $0 = (e^\alpha e^\beta, e^\beta) = (e^\alpha, e^\beta e^\beta) = c_{i, \beta}(e^\alpha, t_i)$.  Therefore $(t_i, e^\alpha) = 0$, as required.

Now, as the condition $\gamma = \alpha + \beta$ is equivalent to $\alpha = \beta + \gamma$, we have $(e^\alpha \cdot e^\beta, e^\gamma ) = (e^\alpha, e^\beta \cdot e^\gamma)$ and the form is indeed Frobenius only if $b_{\alpha, \beta}\lambda_\gamma  = b_{\beta, \gamma}\lambda_\alpha$.  If this condition holds, as the form is symmetric, we also have $(e^\alpha \cdot e^\beta, e^\gamma ) = (e^\beta \cdot e^\alpha, e^\gamma ) = (e^\beta, e^\alpha \cdot e^\gamma)$ and so we obtain the second condition.

Finally, we have $(t_i \cdot t_j, t_k) = (t_i, t_j \cdot t_k)$ and $(t_i \cdot t_j, e^\alpha) = (t_i, t_j \cdot e^\alpha)$. So if $C = \{\0, \1, \alpha, \alpha^c\}$, $|\alpha| = 1$ and $a_{\supp(\alpha), \alpha}=1$, there is no restriction of $(\lambda_{i, \alpha} = (t_i, e^\alpha)$.  This completes the uniqueness proof.
\end{proof}

When we have the structure parameters $\Lambda = (a,b,c)$ with $a \neq 1$, the only possible choice which satisfies the above conditions is $\lambda := \lambda_1 = \dots = \lambda_n$.  Therefore, apart from the exceptional case, $A_C(a,b,c)$ always admits a unique Frobenius form up to scaling.

\begin{proposition}
Let $A_C(\Lambda)$ be a non-degenerate code algebra with regular structure parameters which admits a Frobenius form and let $G:=M:\Aut(C)$ be the group generated by the Miyamoto involutions associated to the $t_i$, $i \in [n]$, and $\Aut(C)$.  Then the form is $G$-invariant if and only if
\[
(t_i, t_i) = (t_{ig}, t_{ig})
\]
for all $g \in G$, $i \in [n]$ and, if $C = \mathbb{F}_2^2$, $\lambda_{\supp(\alpha), \alpha} = \lambda_{\supp(\alpha^c), \alpha^c}$.
\end{proposition}
\begin{proof}
It is clear that the above is necessary for the form to be $G$-invariant, so we must just show that it is sufficient.  Since the structure parameters are regular,
\[
\lambda_{\alpha g} = \frac{c_{i g, \alpha g}}{a_{i g, \alpha g}} \lambda_{i g} = \frac{c_{i, \alpha}}{a_{i, \alpha}} \lambda_i = \lambda_\alpha
\]
for all $\alpha \in C^*$.
\end{proof}

Before we prove our next result we recall the following standard spectral theorem which can be found, for example, in \cite{roman}.

\begin{theorem}\label{spectral}
Let $V$ be a finite-dimensional vector space over $\mathbb{R}$ endowed with a bilinear form which is positive definite and $\phi$ be an endomorphism of $V$ . Then $\phi$ satisfies $( \phi(x), y) = (x, \phi(y))$ for all $x, y \in V$ if and only if $\phi$ is orthogonally diagonalisable.  That is, there is a basis of eigenvectors of $\phi$ with are orthogonal with respect to the form $( \cdot, \cdot)$.
\end{theorem}

\begin{proposition}
Let $A_C$ be a code algebra over $\mathbb{R}$ with a Frobenius form such that $a_{i, \alpha}$ and $c_{i, \alpha}$ have the same sign for all $i \in [n]$ and $\alpha \in C^*$.  Then, for any $a \in A$, $\mbox{ad}_a$ is orthogonally diagonalisable over $K$.  In particular, any idempotent $e \in A$ is semi-simple.
\end{proposition}
\begin{proof}
By Theorem \ref{frob}, the Frobenius form is diagonal with positive coefficients on the diagonal.  Hence, it is positive definite and so by Theorem \ref{spectral}, $\ad_{a}$ is orthogonally diagonalisable.
\end{proof}


\section{The $s$ map} \label{sec:smap}

In this section we will describe a way to find other idempotents of the code algebra.  But before we do, we need one more fact about codes in addition to those in Section \ref{sec:background}.

\begin{lemma}
Let $D$ be a binary linear code. The number of ordered ways of obtaining $\alpha \in D$ as the sum $\beta + \gamma$ where $\beta, \gamma \in D$ is $|D|$.
\end{lemma}
\begin{proof}
Fix $\alpha \in D$ and consider the map $\beta \mapsto \alpha + \beta$.  This map is a bijection from $D$ to $D$, so for $\gamma \in D$ there is exactly one $\beta$ such that $\alpha + \beta = \gamma$.  Hence, considering all different $\alpha$, we see that there are $|D|$ ways to obtain $\gamma$ as an ordered sum.
\end{proof}

If we limit ourselves to $D^* = D -\{\0, \1\}$, the number of ordered ways of obtaining $\gamma \in D^*$ as an ordered sum is
\[
e := \vert D \vert -2(\vert D \vert  - \vert D^* \vert ) = 2\vert D^* \vert - \vert D \vert 
\]

We define a (sub)code $D$ to be \emph{constant weight} if all non-constant codewords have the same weight, i.e.\ $|wt(D^*)|=1$.  (Note that this is a slightly wider definition than is standard.) Observe that if $D$ is a constant weight subcode of $C$, then the sum of any two vectors in $D^*$ is either in $D^*$ or is $\0$ or $\1$.

We introduce the notation $t_D := \sum_{i \in \supp(D)} t_i$, where $D$ is a subcode of $C$ and recall that $t_\alpha = \sum_{i \in \supp(\alpha)} t_i$.

\begin{proposition}\label{smap}
Suppose that $D$ is a constant weight subcode of $C$ and the structure parameters supported on $D^*$ are constant $(a,b,c)$.  Then, for $v \in \mathbb{F}_{2}^{n}$, there exists an idempotent of the form
\[
s(D,v) := \lambda t_D + \mu \sum_{\alpha \in D^{\ast }}\left( -1\right) ^{\left( v ,\alpha \right) }e^{\alpha },
\]
with $\mu, \lambda \in \mathbb{F}$, if and only if
\[
\lambda = \frac{1-be\mu}{2ad}
\]
and $\mu$ satisfies the equation
\[
\left( b^2 e^2+4a^{2}c \vert D^*\vert\frac{d^{3}}{m} \right) \mu^{2} + 2be(ad-1) \mu + 1-2ad=0
\]
where $d$ is the weight of the codewords in $D^*$ and $m := |\supp(D)|$.
\end{proposition}

\begin{remark}
Note that whether an idempotent $s(D,v)$ exists or not, does not depend on $v$, but on the subcode $D$ and the algebra $A$.  However, we can always extend the field of definition for $A$ so that the two above equations for $\lambda$ and $\mu$ have solutions.  In the following, we will often just assume that the field has been taken to be large enough.
\end{remark}

\begin{proof}[Proof of Proposition $\ref{smap}$]
Let $s := s(D,v)$. We begin by multiplying:
\begin{align*}
s\cdot s &= \left( \lambda t_D+\mu \sum_{\alpha \in D^*}( -1) ^{( v ,\alpha ) }e^\alpha \right) \cdot \left( \lambda t_D+\mu \sum_{\alpha \in D^{\ast }}( -1) ^{( v ,\alpha ) }e^\alpha \right) \\
&= \lambda^2 t_D + 2\lambda \mu a\sum_{\alpha \in D^{\ast }}( -1) ^{( v ,\alpha )} \vert \alpha \vert e^\alpha + \mu^2 c \sum_{\alpha \in D^*}t_\alpha \\ 
& \quad +\mu^2 b \sum_{\alpha \in D^*} \sum_{\beta \in D^*, \beta \neq \alpha, \alpha^c} ( -1)^{( v ,\alpha +\beta ) } e^{\alpha +\beta}
\end{align*}

The number of ordered ways of summing two codewords to get a given codeword is $e$.  Also, by part 3 of Lemma \ref{innerC} and since $D$ is constant weight, \allowbreak $\sum_{\alpha \in D^*}t_\alpha = \vert D^* \vert \frac{d}{m} t_D$.  Hence, we get
\[
s\cdot s = \left[ \lambda^2+\mu^2c \vert D^*\vert \frac{d}{m}\right] t_D+ \left[ 2\lambda \mu ad + \mu^{2}be \right] \sum_{\alpha \in D^{\ast }}( -1) ^{( v ,\alpha ) }e^{\alpha }.
\]
So, $s\cdot s=s$ if and only if
\begin{align}
\lambda & = \lambda^2+\mu^2c \vert D^*\vert \frac{d}{m}, \label{lambda} \\
\mu & =  2\lambda \mu ad + \mu^{2}be.  \label{mu}
\end{align}
As we may assume $\mu \neq 0$, equation (\ref{mu}) implies
\begin{align*}
1 & =  2\lambda  ad + \mu be, \\
\lambda  &= \frac{1-be\mu}{2ad}.
\end{align*}
Substituting this value of $\lambda$ in equation (\ref{lambda}) we obtain the required quadratic equation in $\mu$.
\end{proof}

\begin{remark}
Observe that $s(D,v) = s(D,w)$ if and only if $(v, \alpha) = (w, \alpha)$, for all $\alpha \in D^*$. That is, if $v$ and $w$ lie in the same coset of $D^\perp$.  Hence, if $D$ is a subcode for which $s$-map idempotents exist, for a choice of root $\mu$, the idempotents $s(D,V)$ are in bijection with the cosets of $D^\perp$.
\end{remark}
\begin{remark}
Using the previously introduced notation, $D_i := \{ \alpha \in D : (\alpha, v) = i \}$, we can rewrite the idempotent coming from the $s$-map as
\[
s(D,v) = \lambda t_D + \mu \big( \sum_{\alpha \in D_0^*}e^{\alpha } - \sum_{\beta \in D_1^*} e^\beta\big).
\]
\end{remark}



\subsection{Small idempotents}\label{sec:smallidem}

While for any given specific choice of code $C$ and subcode $D \leq C$, we may calculate the eigenvalues and eigenvectors for the idempotent $s(D,v)$, and hence its fusion law, it is difficult to give a general result.  However,  for some subcodes $D$ this is possible.  In particular, we always have the subcode $D = \langle \alpha \rangle$ for some $\alpha \in C^*$ and so, provided our field is large enough, we get an idempotent $s(D, v)$. We call this a \emph{small idempotent}. Moreover, we show that under some additional conditions on the code, the algebra is generated by such small idempotents, making $A$ an axial algebra.

Let $D := \{ \bf 0, \alpha \}$ for some $\alpha \in C^*$. We assume that the structure parameters supported on $D$ do not depend on the toral or codeword elements. Note that this just means that $a_\alpha := a_{i, \alpha} = a_{j, \alpha}$ and $c_\alpha := c_{i, \alpha} = c_{j, \alpha}$ for $i \in \supp(\alpha)$.  Now, $e = 2|D^*| - |D| = 0$, $d=m = |\alpha|$ and so the equations in Proposition \ref{smap} are dramatically simplified. In particular, we find that
\[
\lambda = \frac{1}{2a_\alpha|\alpha|} \quad \mbox{and} \quad \mu^2 = \frac{\lambda - \lambda^2}{c_\alpha}
\]
Provided the field is large enough, we get two idempotents $e_\pm := \lambda t_\alpha \pm \mu e^\alpha$.  If the field is $\mathbb{R}$, then the discriminant in the quadratic for $\mu$ implies we have solutions if $a_\alpha c_\alpha > \frac{c_\alpha}{2|\alpha|}$. (Note that we write the inequality in such a way to deal with the two cases $c_\alpha>0$ and $c_\alpha <0$ simultaneously.)  Note that if $a_\alpha = \frac{1}{2|\alpha|}$, then $\lambda =1$, $\mu=0$ and so $e_\pm = t_\alpha$.  We wish to rule out this degenerate case, so we will assume that $a_\alpha \neq \frac{1}{2|\alpha|}$ for the remainder of this section.

We get the following corollary to Proposition \ref{smap}.

\begin{corollary}
Let $A_C$ be a non-degenerate code algebra.  Suppose that $S$ is a generating set for the code $C$ such for all $\alpha \in S$ the small idempotents corresponding to $\alpha$ exist.  Then $A$ is generated by the idempotents $t_i$, for $i \in [n]$, and the small idempotents corresponding to $S$.
\end{corollary}

Note that our main theorem below is for a constant weight code.  The more complicated case of an arbitrary code (not constant weight), will be dealt with in an upcoming paper \cite{codealg2}.

\begin{theorem}\label{smallidem}
Let $A_C$ be a non-degenerate code algebra on a constant weight code $C$ such that
\begin{align*}
a  &:= a_{i, \beta} && \mbox{for all } i \in \supp(\beta), \beta \in C^*\\
b_{\alpha, \beta} &= b_{\alpha, \gamma} && \mbox{for all } \beta,\gamma \in C^* -\{ \alpha, \alpha^c\} \\
b_{\alpha^c, \beta} &= b_{\alpha^c, \gamma} &&\mbox{for all } \beta,\gamma \in C^* -\{ \alpha, \alpha^c\} \\
c_\beta &:= c_{i, \beta} && \mbox{for all } i \in \supp(\beta), \beta \in C^*
\end{align*}
Suppose that for $\alpha \in C^*$, the small idempotents $e_\pm$ exist and $a \neq \frac{1}{2|\alpha|}, \frac{1}{3|\alpha|}$.  Then, $e_\pm$ are primitive axes for the fusion law given in Table $\ref{tab:small}$.

\begin{table}[!htb]
\setlength{\tabcolsep}{4pt}
\renewcommand{\arraystretch}{1.5}
\centering
\begin{tabular}{c||c|c|c|c|c|c}
 & $1$ & $0$ & $\lambda$ & $\lambda -\frac{1}{2}$ & $\nu_+$ & $\nu_-$ \\ \hline \hline
$1$ & $1$ &  & $\lambda$ & $\lambda -\frac{1}{2}$ & $\nu_+$ & $\nu_-$ \\ \hline
$0$ &  & $0$ &  &  & $\nu_+$ & $\nu_-$ \\ \hline
$\lambda$ & $\lambda$ &  & $1, \lambda, \lambda -\frac{1}{2}$ &  & $\nu_-$ & $\nu_+$ \\ \hline
$\lambda -\frac{1}{2}$ & $\lambda -\frac{1}{2}$&  &  & $1, \lambda -\frac{1}{2}$ & $\nu_+$ & $\nu_-$ \\ \hline
$\nu_+$ & $\nu_+$ & $\nu_+$  & $\nu_-$ & $\nu_+$ & $1,0,\lambda, \lambda -\frac{1}{2}, \nu_+ , \nu_-$ & $0, \lambda $ \\ \hline
$\nu_-$ & $\nu_-$ & $\nu_-$  & $\nu_+$ & $\nu_-$ & $0, \lambda $ & $1,0,\lambda, \lambda -\frac{1}{2}, \nu_+, \nu_-$
\end{tabular}
where $\nu_\pm := \frac{1}{4} \pm \mu b$.
\caption{Fusion table for small idempotents}\label{tab:small}
\end{table}
\end{theorem}

\begin{remark}
We note that generically, i.e.\ when the eigenvalues do not coincide, the fusion table given in Table \ref{tab:small} satisfies the Seress condition.
\end{remark}

The proof of Theorem \ref{smallidem} will proceed via a series of lemmas.  Throughout the remainder of this section,  let $e = \lambda t_\alpha + \mu e^\alpha$ for some $\alpha \in C^*$, where $\lambda = \frac{1}{2a_\alpha|\alpha|}$ and $\mu^2 = \frac{\lambda - \lambda^2}{c_\alpha}$.  We will list all the eigenvectors and eigenvalues for $e$; we begin by writing down some obvious ones.

\begin{lemma}\label{eig}
For an arbitrary code $C$, we have the following eigenvectors for $e$.
\begin{enumerate}
\item[$1.$] Eigenvalue $0$\textup{:}
\begin{enumerate}
\item[\textup{i.}] $t_i$ \qquad such that $i \notin \supp(\alpha)$,
\item[\textup{ii.}] $e^{\alpha^c}$\qquad provided ${\bf 1 } \in C$.
\end{enumerate}
\item[$2.$] Eigenvalue $\lambda$\textup{:}
\[
t_i - t_j \qquad \mbox{where } i, j \in \supp(\alpha).
\]
\item[$3.$] Eigenvalue $\lambda -\frac{1}{2}$\textup{:}
\[
2\mu c_\alpha \, t_\alpha - e^\alpha.
\]
\end{enumerate}
\end{lemma}
\begin{proof}
These are straightforward calculations.
\end{proof}

%

\begin{lemma}\label{pairedeig}
Suppose $C$ is constant weight and let $\beta \in C^*$ such that $\beta \neq \alpha, \alpha^c$, $a_\beta := a_{i, \beta}$ for all $i \in \supp(\beta)$, $a_\beta = a_{\alpha+\beta}$, $b_{\alpha, \beta} = b_{\alpha, \alpha + \beta}$.  Then,
\begin{itemize}
\item[$1.$] $e^\beta - e^{\alpha + \beta}$ is a $(\frac{a_\beta}{4a_\alpha} - \mu b_{\alpha, \beta})$-eigenvector for $e$
\item[$2.$] $e^\beta + e^{\alpha + \beta}$ is a $(\frac{a_\beta}{4a_\alpha} + \mu b_{\alpha, \beta})$-eigenvector for $e$
\end{itemize}
\end{lemma}
\begin{proof}
We prove the first claim, the second is similar.
\begin{align*}
(\lambda t_\alpha + \mu e^\alpha)(e^\beta - e^{\alpha+\beta}) &= (\lambda a_\beta |\supp(\alpha) \cap \supp(\beta)| - \mu b_{\alpha, \alpha+\beta}) e^\beta \\
& \qquad - (\lambda a_{\alpha+\beta} |\supp(\alpha) \cap \supp(\alpha+\beta)| - \mu b_{\alpha, \beta})e^{\alpha+\beta}
\end{align*}
Since $C$ is constant weight and $|\alpha+\beta| = |\alpha| + |\beta| - 2 | \supp(\alpha) \cap \supp(\beta)|$, it is clear that all distinct codewords in $C^*$ which are not complements intersect in a set of size $\frac{|\alpha|}{2}$.  Using this and our assumptions on the structure parameters, we see that the two coefficients are both $\frac{a_\beta}{4a_\alpha} - \mu b_{\alpha,\beta}$.
\end{proof}

\begin{proof}[Proof of Theorem $\ref{smallidem}$]
The only two eigenvectors which could be scalar multiples of one another are $e = \lambda t_\alpha + \mu e^\alpha$ and $2\mu c \, t_\alpha - e^\alpha$.  This happens if and only if
\begin{align*}
\lambda &= -2\mu^2 c \\
 &= -2(\lambda - \lambda^2) \\
0 &= \lambda(2\lambda -3)
\end{align*}
which is equivalent to $a= \frac{1}{3|\alpha|}$.  We note that this holds precisely when the $\lambda -\frac{1}{2}$-eigenspace coincides with the $1$-eigenspace.

Provided $a \neq \frac{1}{3|\alpha|}$, it is easy to see all the eigenvectors listed are linearly independent.  If ${\bf 1 } \in C$, we have listed $1+(\frac{n}{2}+1)+ (\frac{n}{2} -1) + 1+ \frac{|C^*| -2}{2} + \frac{|C^*| -2}{2} = n + |C^*|$ elements and if ${\bf 1 } \notin C$, we have listed $1+\frac{n}{2}+ (\frac{n}{2} -1) + 1+ \frac{|C^*| -1}{2} + \frac{|C^*| -1}{2} = n + |C^*|$ elements.  In both cases this is $n + |C^*|$ which is the dimension of the algebra.

It remains to calculate the fusion table for the small idempotents.  This is somewhat long, but easy calculation. As an example, we provide one such calculation here for $\nu_+\star \nu_+$:
\begin{align*}
(e^\beta + e^{\alpha+\beta})^2 &= c_\beta t_\beta + c_{\alpha+\beta} t_{\alpha+\beta} + 2b_{\alpha, \alpha+\beta} e^\alpha \\
(e^\beta + e^{\alpha+\beta})(e^\gamma + e^{\alpha+\gamma}) &= (b_{\beta,\gamma} + b_{\alpha+\beta, \alpha+\gamma})e^{\beta+\gamma} + (b_{\alpha+\beta,\gamma} + b_{\alpha, \alpha+\gamma})e^{\alpha+\beta+\gamma}
\end{align*}
the first of these is contained in $A_1 \oplus A_0 \oplus A_\lambda \oplus A_{\lambda - \frac{1}{2}}$, whilst the second is contained in $A_{\nu_+} \oplus A_{\nu_-}$.  (Note that if $b_{\beta, \gamma} = b_{\alpha + \beta, \gamma}$ for all $\beta, \gamma \in C^* -\{ \alpha, \alpha^c\}$, then the second expression above is in fact contained in $\nu_+$).
\end{proof}

For $A$ to be an axial algebra with axes equal to the small idempotents, we need some additional conditions on the code $C$.  We define the $m$-fold juxtaposition of an $[n,k,d]$-code $C$ with generating matrix $G$ to be the $[mn,k,md]$-code with generating matrix given by $m$ copies of $G$:
\[
\left( G | \dots | G \right ).
\]
This is independent of the choice of generating matrix.

\begin{lemma}
Let $C$ be a constant weight code.
\begin{enumerate}
\item[$1.$] If ${\bf 1 } \notin C$, then $C$ is equivalent to the $m$-fold juxtaposition of a simplex code \emph{(}i.e.\ a dual of a Hamming code\emph{)} with parameters $[2^r-1,r,2^{r-1}]$.
\item[$2.$] If ${ \bf 1 } \in C$, then $C$ is equivalent to the $m$-fold juxtaposition of a first order Reed-Muller code with parameters $[2^r, r+1, 2^{r-1}]$.
\end{enumerate}
Moreover, there is a bijection between the first and second types above.
\end{lemma}
\begin{proof}
The first part is the result from \cite{constantwt}.  We describe the bijection from which the second part will follow.

Let $C$ be a code of the first type.  Then it has parameters $[m(2^r-1), r, m2^{r-1}]$.  We lengthen the code by adding $m$ zero columns to the generating matrix and then augmenting by adding the $1$ codeword.  This produces an $[m2^r, r+1, m2^{r-1}]$ code.  Since the minimal distance is half the length and it contains the $1$ codeword, it must be of the second type.  Conversely, if $C$ is a code of the second type of length $n$, then $C_0$ is a code which does not contain the $1$ codeword and has one weight of codeword with weight $\frac{n}{2}$.  After removal of an zero columns in the generating matrix, it must be of the first type.  We note that the bijection takes a simplex code to a first order Reed-Muller code.
\end{proof}

\begin{definition}
A binary linear code $C \subset \mathbb{F}_2^n$ is called \emph{projective} if $C^\perp$ has minimum distance at least $3$.
\end{definition}

\begin{lemma}\label{Cproj}
Let $C$ be a binary linear code.  Then $C$ is projective if and only if for all $i \in 1, \dots, n$, there exists a set of codewords $S$ such that
\[
\{i\} = \bigcap_{\alpha \in S}  \supp(\alpha)
\]
\end{lemma}
\begin{proof}
Suppose that the above property holds.  Then, for all $i$, there exists a codeword $\alpha \in C$ with $\alpha_i = 1$ and hence $C^\perp$ has no codewords of weight $1$.  Moreover, for all $i \neq j$, there exists $\alpha \in C$ such that $\alpha_i \neq \alpha_j$.  Hence, $C^\perp$ has no codeword of weight $2$ and $C$ is projective.

Conversely, suppose that the above property does not hold for some $i = 1, \dots , n$. Either there does not exist a codeword in $C$ supported on $i$, and hence $C^\perp$ contains a codeword of weight one, or there exists $i \neq j$ such that for every codeword $\alpha \in C$, $\alpha_i = \alpha_j$, and hence $C^\perp$ has a codeword of weight two.  In any case, $C$ is not projective.
\end{proof}

Note that a constant weight code $C$ is projective if and only if it is not a $m$-fold juxtaposition.  That is, if and only if it is a simplex or first order Reed-Muller code.

\begin{theorem}\label{smallaxial}
Let $C$ be a simplex or first order Reed-Muller code and $A_C(a, b, c)$ be a non-degenerate code algebra with $a \neq \frac{1}{2|\alpha|}, \frac{1}{3|\alpha|}$.  Suppose that the field is large enough so that the small idempotents exist.  Then, $A$ is an axial algebra with respect to the small idempotents $\lambda t_\alpha \pm \mu e^\alpha$, $\alpha \in C^*$, with fusion law given in Table $\ref{tab:small}$.
\end{theorem}

\begin{proof}
By Theorem \ref{smallidem}, the small idempotents $e_\pm$ are primitive and semi-simple. Clearly, $e^\alpha$ and $t_\alpha$ are in the span of $e_\pm = \lambda t_\alpha \pm \mu e^\alpha$.  As $C$ is a projective code, there exist $\alpha_1, \dots, \alpha_n \in C^*$, such that
\[
t_i = t_{\alpha_1}\dots t_{\alpha_n}
\]
So, the small idempotents generate the algebra.
\end{proof}

We note that the fusion law in Table \ref{tab:small} may be simplified when some of the eigenvalues coincide.  If the fusion law degenerate into those of the $t_i$, then we may drop the requirement for the code to be a simplex or a first-order Reed-Muller code and instead require it to be an arbitrary constant weight code (i.e. not necessarily projective).  On the other hand, for a projective code $C$ which does not have constant weight, the $s$-map construction will give other types of idempotent depending on the subcode $D$.  Given such a subcode $D$, provided the conjugates of the subcode contain enough codewords to generate the code, the corresponding idempotents will generate the algebra then one has a theorem as above and the algebra is an axial algebra.

\section{Examples} \label{sec:examples}

In this section we give some examples, the most important of which is the Hamming code example.  We wrote a package in {\sc magma} \cite{magma} to assist in some of the calculations for this section. All of the examples we give here are over $\mathbb{R}$.

\subsection{The complete code $\mathbb{F}_2^2$}

Let $C = \mathbb{F}_2^2$.  Let the structure parameters $\Lambda = (a,b,c)$, with $a,b,c \neq 0$, not depend on the codeword or toral element.  Then $A := A_C(a,b,c)$ is a non-degenerate code algebra.  We note that if $D = \{ { \bf 0, 1, } \alpha, \alpha^c \}$ is a code (or a subcode of some code $D'$), then the code algebra $A_D$ is similar to $A_C$ and differs by the possible addition of some extra toral idempotents.

Let $\alpha = (1,0)$ and so $\alpha^c = (0,1)$. By Theorem \ref{frob}, $A$ has a Frobenius form whose Gram matrix is $\diag(1,1,\tfrac{a}{c}, \tfrac{a}{c})$, with respect to the standard basis $\{ t_1, t_2, e^{\alpha}, e^{\alpha^c} \}$.

As noted in Lemma \ref{non-simpleA}, $A$ has precisely two non-trivial ideals and it is easy to see that these are orthogonal with respect to the form, giving us a decomposition
\[
A = \langle t_1, e^\alpha \rangle \oplus \langle t_2, e^{\alpha^c} \rangle
\]
This is also orthogonal in the sense that elements from the two different ideals multiply to give $0$.  So, to find any other idempotents in $A$ it is enough to consider idempotents which lie in either of the two isomorphic ideals,  as any idempotent $e$ of $A$ is a sum of one idempotent from each ideal $e = e_1 + e_2$, where we allow $e_i = 0$.

\begin{lemma}
The idempotents of the subalgebra $A_1 := \langle t_1,e^\alpha \rangle$ are $t_1$ and, when $2ac >c$,
\[
e_\pm := \tfrac{1}{2a}\left( t_{1}  \pm \sqrt{\tfrac{( 2a-1)}{c}} e^\alpha \right).
\]
Furthermore, $e_\pm$ have a $\tfrac{1-a}{2a}$-eigenvector given by
\[
v_\pm := \pm \tfrac{1}{a}\sqrt{(2a-1)c} \,t_1 -e^\alpha.
\]
\end{lemma}
\begin{proof}
It is a simple calculation to show that the only idempotents must be of the form above and these can only exist when the expression under the square root is a positive real. That is, when $2ac >c$.  Indeed, the fact that these are idempotents follows from Section \ref{sec:smallidem}.  Using Lemma \ref{eig}, we see that $v_\pm$ are the required eigenvectors.
\end{proof}

\begin{corollary}
In the subalgebra $A_1 := \langle t_1,e^\alpha \rangle$, $e_\pm$ has the same fusion law as $t_1$ if and only if $a = -1$ and $c <0$.  In this case, $A_1$ is a $2$-dimensional axial algebra with three idempotents.
\end{corollary}
\begin{proof}
We require the eigenvalue $\tfrac{1-a}{2a} = a$, giving $2a^2+a-1 = (2a-1)(a+1) =0$.  Since $a = \frac{1}{2}$ would imply that $e_\pm = t_1$, we are left with $a=-1$.  Then, for $e_\pm$ to exists, we require $c <0$.
\end{proof}

\subsection{Even weight vectors in $\mathbb{F}_2^3$}

Consider the even weight code in $\mathbb{F}_2^3$
\[
C=\left\{ \left( 0,0,0\right) ,\left( 0,1,1\right) ,\left( 1,0,1\right) ,\left( 1,1,0\right) \right\}
\]
and let $A := A_C$ be the code algebra with non-zero structure parameters $(a,b,c)$.  Then $A$ is a non-degenerate code algebra with a Frobenius form with Gram matrix $\diag(1,1,1,\frac{c}{a}, \frac{c}{a}, \frac{c}{a})$ in the usual basis.

Using Proposition \ref{smap}, we see that, in addition to the small idempotents we can get from the three subcodes of the form $\{ \bf 0, \alpha \}$, we can get $s$-map idempotents from the code $C$.  In particular, if we pick $(a,b,c) = (\frac{1}{2}, \frac{1}{2}, 1)$ then the fusion law for the small idempotents degenerates to be the same as for the toral idempotents $t_i$.  Since these idempotents generate the algebra, $A$ is an axial algebra.  Moreover, it is an axial algebra of Jordan type $\frac{1}{2}$ (i.e.\ the eigenvalues are $1$, $0$, $\frac{1}{2}$).  There are 6 such small idempotents given by
\[
e_{\pm \alpha} := \tfrac{1}{2} t_\alpha \pm \tfrac{1}{2}e^\alpha
\]

Moreover, the $s$-map idempotents from $C$ also have the same fusion law.  In total, we have four primitive such idempotents coming from choosing $v$ to be $(0,0,0), (1,0,0), (0,1,0), (0,0,1)$.  Other choices of $v$ give the same results.  These are given by
\[
s(C,v) = \tfrac{1}{3}t + \tfrac{1}{3} \sum_{\alpha \in C*} (-1)^{(v,\alpha)} e^\alpha
\]
where $t = \sum_{i=1}^3 t_i$.

The group spanned by the Miyamoto involutions associated with these 13 involutions generate an infinite group.  Indeed the composition of any two of the Miyamoto involutions associated to $s(C,v)$ coming from the code $C$ gives an element $g$ of infinite order.  The orbit of, for example, $t_1$ under $g$ is an infinite set of idempotents.  If however, we restrict to the $t_i$ and small idempotents, then these $9$ axes are closed under the action of their associated Miyamoto involutions.

In fact, this example is an example of the Jordan algebra of symmetric matrices with the product $a \circ b = \tfrac{1}{2}(ab - ba)$.  To see this, let $E_{i,j}$ be the matrix which is all zero except for a $1$ in the $i,j$ position and observe that $t_i \mapsto E_{i,i}$ and $e^\alpha \mapsto E_{i,j} + E_{j,i}$, where $\{i,j\} = \supp(\alpha)$ gives the desired map.  This also works when $C$ is the even code of length $n = 4$, but not for larger $n$.

\subsection{Hamming code}\label{sec:hamming}

Let $C = H_8$ be the extended Hamming code of length $8$ and pick $(a,b,c) = (\tfrac{1}{4}, \tfrac{1}{2},1)$ to be the structure parameters.  Then, $A := A_C$ is a $22$-dimensional non-degenerate code algebra with a Frobenius form.  We note that the structure parameters are the same as those which come from a VOA.

Again, using Proposition \ref{smap}, we get more idempotents which are of the form
\[
s(C,v) = \tfrac{1}{8}t + \tfrac{1}{8} \sum_{\alpha \in C*} (-1)^{(v,\alpha)} e^\alpha
\]
where $t = \sum_{i=1}^8 t_i$.  By choosing $v \in \mathbb{F}_2^8 - C$ to have odd weight, we get a set of eight mutually orthogonal idempotents.  That is, a torus.  Similarly, we obtain another torus by using even weight vectors in $\mathbb{F}_2^8 - C$.  Together with the standard torus, this gives three distinct tori.  Moreover, all these idempotents have the same fusion law and they generate the algebra, showing that $A$ is an axial algebra of Jordan type $\frac{1}{4}$.  Since they also span the algebra, $A$ is $1$-closed.

The Miyamoto involution associated with the the axis in one torus permutes the other two tori.  In particular, in this example, unlike the even code example above, the set of $24$ axes is closed under the action of their Miyamoto involutions.

The automorphism group of the code acts strongly transitively, hence its action is induced faithfully on $A$.  The automorphism group of the Hamming code is $2^3{:}PSL_3(2)$.  Together with the Miyamoto involutions we see a group with shape
\[
G := 2^6{:}(PSL_3(2) \times S_3).
\]
By Theorem \ref{motivatingthm}, we see that $A$ can be embedded in the VOA $V_{H_8}$ constructed from the Hamming code $H_8$.  Moreover, the non-zero elements of $A$ lie in the weight $2$ graded part of $V_{H_8}$.  We note that the automorphism group $G \leq \Aut(A_{H_8})$ is the same group as the full automorphism group of $V_{H_8}$ \cite{HammingVOA}.


\appendix

\section{Appendix: Links to VOAs}\label{sec:motivation}

In this appendix, we outline the construction of code VOAs which provides part of the motivation for the definition of code algebras.  For more details we refer the reader to \cite{M2}. We do not attempt to define Virasoro algebras or VOAs. Throughout, let $C \subseteq \mathbb{F}_2^n$ be an \emph{even} binary linear code of length $n$ (i.e. all codewords in $C$ have even weight). 

The Virasoro algebra $\Vir$ with central charge $\frac{1}{2}$ has three irreducible lowest weight modules $L(\frac{1}{2},0)$, $L(\frac{1}{2},\frac{1}{2})$ and $L(\frac{1}{2},\frac{1}{16})$ corresponding to the three possible weights $0$, $\frac{1}{2}$ and $\frac{1}{16}$. Let
\[ M =  L(\tfrac{1}{2},0) \otimes L(\tfrac{1}{2},\tfrac{1}{2}). \]
The fusion table of the lowest weight modules are known:
\begin{align*}
L(\tfrac{1}{2},0) \times L(\tfrac{1}{2},0) &= L(\tfrac{1}{2},0), \\
L(\tfrac{1}{2},0) \times L(\tfrac{1}{2},\tfrac{1}{2}) &= L(\tfrac{1}{2},\tfrac{1}{2}) \times L(\tfrac{1}{2},0) = L(\tfrac{1}{2},\tfrac{1}{2}), \\
L(\tfrac{1}{2},\tfrac{1}{2}) \times L(\tfrac{1}{2},\tfrac{1}{2}) &= L(\tfrac{1}{2},0).
\end{align*}
So, $M$ has a natural $\ZZ_2$-grading. Since the two lowest weight modules are also $\ZZ^+$-graded, $M$ is $\frac{1}{2}\ZZ^+$-graded with $M_{\bar{0}} = L(\frac{1}{2},0)$ and $M_{\bar{1}} = L(\frac{1}{2},\frac{1}{2})$, where $\bar{0}$ denotes the integer graded part and $\bar{1}$ denotes the half-integer graded part. This makes $M$ into a \emph{super vertex operator algebra} (SVOA) with vertex operator $Y:M \otimes M \to M((z))$, where $M((z))$ is the space of formal Laurent series with coefficients in $M$.

We denote by $w$ the Virasoro element of $M$; this lies in $M_0$. We fix a lowest weight vector $q$ with respect to $w$ for the eigenvalue $\frac{1}{2}$.  So, by definition, we have:
\begin{alignat*}{2}
w_{1}q &= L_0q &&= \tfrac{1}{2} q, \\
w_{n+1}q &= L_n q &&= 0,  \qquad \mbox{for all } n \geq 1,
\end{alignat*}
where $L_n$ are the generators in the Virasoro algebra.  By direct computation \cite[(4.9)]{M3}, we obtain:
\begin{alignat*}{3}
q_{-2}q &= 2w, \qquad & q_{-1}q &= 0,\qquad & q_0q &= 1, \\
q_{-1}1 &= q, & q_nq &= 0, & \mbox{for all } &n \geq 1.
\end{alignat*}

We now consider a tensor product of $n$ copies of the SVOA $M$
\[
\hat{M} := M^1 \otimes \dots \otimes M^n,
\]
where $M^i \cong M$.  This produces a SVOA with vertex operator $\hat{Y}$ given by
\[
\hat{Y}(\hat{M},z) = \hat{Y}(M^1 \otimes \dots \otimes M^n, z) := Y^1(M^1,z) \otimes \dots \otimes Y^n(M^n,z).
\]
Given a codeword $\alpha \in C$, set
\[
\hat{M}_\alpha := \bigotimes_{i=1}^n M^i_{\overline{\alpha_i}}.
\]
For example, $\hat{M}_{(0, \dots, 0)} = L(\tfrac{1}{2},0) \otimes\dots \otimes L(\tfrac{1}{2},0)$ and $\hat{M}_{(1, \dots, 1)} = L(\tfrac{1}{2},\tfrac{1}{2}) \otimes\dots \otimes L(\tfrac{1}{2},\tfrac{1}{2})$.  We have the following:

\begin{lemma}{\normalfont \cite[Lemma 4.2]{M3}}
Let $u \in \hat{M}_\alpha$, $v \in \hat{M}_\beta$ with $\alpha, \beta \in C$.  Then,
\[
u_mv \in \hat{M}_{\alpha + \beta}, \qquad \mbox{for all } m \in \ZZ.
\]
\end{lemma}
Since $\hat{M}$ is an SVOA, it has supercommutativity; however, we want a VOA with commutativity of the operator. To correct this, we take a central extension of $C$ by $\pm 1$.  Let $\nu_i := (0, \dots, 0,1,0,\dots, 0)$ with a $1$ in the $i$\textsuperscript{th} position.  We define formal elements $e^{\nu_i}$ which satisfy $e^{\nu_i}e^{\nu_i} = 1$ and $e^{\nu_i}e^{\nu_j} = -e^{\nu_j}e^{\nu_i}$ for $i \neq j$.  For words $\alpha = \nu_{j_1} +\dots + \nu_{j_k}$ with $j_1 < \dots < j_k$, define
\[
e^\alpha = e^{\nu_{j_1}}\dots e^{\nu_{j_k}}
\]
Then, we have the following
\begin{lemma}{\normalfont \cite[Lemma 4.4]{M3}}
For $\alpha$ and $\beta$,
\[
e^\alpha e^\beta = (-1)^{(\alpha, \beta) + \vert \alpha \vert \vert \beta \vert} e^\beta e^\alpha
\]
where $( \cdot, \cdot)$ denotes the usual dot product.
\end{lemma}

Combining this with the previously defined SVOA, we obtain the VOA
\[
M_C := \bigoplus_{\alpha \in C} M_\alpha,
\]
where $M_\alpha := (\hat{M}_\alpha \otimes e^\alpha)$, which has vertex operator
\[
Y(v\otimes e^\alpha, z) := \hat{Y}(v,z)\otimes e^\alpha.
\]

\begin{theorem}{\normalfont \cite[Theorem 4.2]{M3}}
When $C$ is an even binary linear code, $(M_C, Y, {\bf w}, \1)$ is a simple VOA, with Virasoro element
\[
{\bf w} = \sum_{i=1}^n (1 \otimes \dots \otimes 1 \otimes w^i \otimes 1 \otimes \dots \otimes 1) \otimes e^0
\]
where $w^i$ is the Virasoro element for $M^i$ and vacuum element
\[
\1 = (1 \otimes \dots \otimes 1) \otimes e^0.
\]
\end{theorem}

For our purposes, we want to consider certain elements of this VOA $M_C$.  For $1 \leq i \leq n$ and $\alpha \in C$, define
\begin{align*}
{\bf w}^i &= (1 \otimes \dots \otimes 1 \otimes w^i \otimes 1 \otimes \dots \otimes 1) \otimes e^0 \\
q^\alpha &= \left( \bigotimes_{i=1}^n q^{\alpha_i} \right) \otimes e^\alpha
\end{align*}
where we understand $q^1 = q$ and $q^0 = 1$.  Note that, since $q$ has grading $\frac{1}{2}$, $q^\alpha$ has grading $\frac{\vert \alpha \vert}{2}$.  In order to multiply certain elements in $M_C$, we first look at multiplication in $M$.

\begin{lemma}
In $(M,y,1,w)$, we have
\begin{enumerate}
\item[$1.$] $Y(1,z)1 = 1$
\item[$2.$] $Y(1,z)q = q$
\item[$3.$] $Y(q,z)1 = q + \sum_{k \geq 1} q_{-k-1}1 z^k$
\item[$4.$] $Y(q,z)q = 1z^{-1} + 2w z + \sum_{k \geq 2} q_{-k-1}q z^k$   \qed
\end{enumerate}
\end{lemma}

Using the above lemma, we can now calculate some products in $M_C$.

\begin{proposition}\label{VOAprod}
Let $\alpha, \beta \in C$.  Then
\begin{enumerate}
\item[$1.$] ${\bf w}^i_{\phantom{i}1} {\bf w}^j = 2{\bf w}^i \delta_{ij}$
\item[$2.$] ${\bf w}^i_{\phantom{i}1} q^\alpha = \begin{cases} \frac{1}{2} q^\alpha  & \mbox{if } i \in \supp(\alpha) \\
  0 & \mbox{otherwise} \end{cases}$
\item[$3.$] $q^\alpha_{\phantom{\alpha}\vert \alpha \cap \beta \vert -1} q^\beta = q^{\alpha+\beta}$
\item[$4.$] $q^\alpha_{\phantom{\alpha}\vert \alpha \vert -3} q^\alpha = 2 \sum_{\supp(\alpha)} {\bf w}^i $
\item[$5.$] $q^\alpha_{\phantom{\alpha}n} q^{\alpha^c} = 0$ \quad for any $n \geq 0$
\end{enumerate}
\end{proposition}
\begin{proof}
The first part is obvious as ${\bf w}$ is a conformal vector, while the proofs for the fourth and fifth parts are similar to the third.

For the second part, consider $Y({\bf w}^i, z)q^\alpha$.  If $i \notin \supp(\alpha)$, then in the $i$\textsuperscript{th} position we have $w_1 1 = 0$ and hence the result follows.  Suppose $i \in \supp(\alpha)$. Then in all positions $j \neq i$ we have $Y(1,z)u = u$ but in position $i$ we have $Y(w,z)q = \frac{1}{2}q z^{-2} + \sum_{k \geq 1} L_{-k}q z^{k-2}$. Suppose, without loss of generality, that $i=1$ and $\supp(\alpha) = \{1, \dots, k\}$, then we have:
\begin{align*}
Y({\bf w}^1, z)q^\alpha &= (\tfrac{1}{2}q z^{-2} + \sum_{k \geq 1} L_{-k}q z^{k-2}) \otimes q \otimes \dots \otimes q \otimes 1 \otimes \dots \otimes 1 \\
 &=(\tfrac{1}{2}q \otimes q \otimes \dots \otimes q \otimes 1 \otimes \dots \otimes 1) z^{-2} \\
& \qquad  + \sum_{k \geq 1} (L_{-k} q \otimes q \otimes \dots \otimes q \otimes 1 \otimes \dots \otimes 1) z^{k-2}
\end{align*}
The $1$-product is given by reading of the coefficient of the $-1-1 = -2$ power of $z$ above, which is $\frac{1}{2} q^\alpha$.

For the third, consider $Y(q^\alpha, z)q^\beta$.  In $\supp(\alpha)-\supp(\beta)$ positions we have $Y(q,z)1 = q + \sum_{k \geq 1} q_{-k-1}1 z^k$, in $\supp(\beta)-\supp(\alpha)$ positions we have $Y(1,z)q = q$ and for the intersection $\supp(\alpha) \cap \supp(\beta)$ we have $Y(q,z)q = 1z^{-1} + 2w z + \sum_{k \geq 2} q_{-k-1}q z^k$.  Let $I := \vert \alpha \vert$, $J := \vert \beta \vert$ and $K := \vert \supp(\alpha + \beta) \vert$.  Without loss of generality, we assume that $\alpha$ has a 1 in the first $I$ positions and $\beta$ has a 1 in positions $ I-K +1$ to $I -K + J$.
\begin{align*}
Y(q^\alpha, z)q^\beta &= \bigotimes_{i=1}^{I-K} (q + \sum_{k \geq 1} q_{-k-1}1 z^k) \\
& \qquad \otimes \bigotimes_{i=1}^{K} (1z^{-1} + 2w z + \sum_{k \geq 2} q_{-k-1}q z^k) \\
& \qquad \otimes \left(\bigotimes_{i=1}^{J-K} q \right) \otimes 1 \otimes  \dots \otimes 1 \\ 
&= (\bigotimes_{i=1}^{I-K}q \otimes \bigotimes_{i=1}^{K} 1 \otimes \bigotimes_{i=1}^{J-K} q \otimes 1 \otimes \dots \otimes 1) z^{-K}\\
& \qquad + \mbox{higher order terms}
\end{align*}
So, we see that the $K-1 = \vert \alpha \cap \beta \vert -1$ product is $q^{\alpha + \beta}$ as required.
\end{proof}

Now that we have given some particular products in $M_C$, we may give the following motivating theorem.  In Section \ref{sec:codealg}, we will see that this gives an example of a code algebra.

\begin{theorem}\label{motivatingthm}
Let $C$ be an even binary linear code of length $n$ and $M = M_C$ be the corresponding code VOA.  Suppose $A$ is a submodule spanned by
\begin{align*}
t_i &:= \tfrac{1}{2}{\bf w}^i \\
e^\alpha &:= \lambda q^\alpha
\end{align*}
where $i \in 1, \dots, n$ and $\alpha \in C^*$.  We define a commutative multiplication given by
\begin{align*}
t_i t_j &=  \mathrlap{(t_i)_1 t_j} \phantom{(t_i)_1 e^\alpha} = t_i \delta_{ij}\\
t_i e^\alpha &= (t_i)_1 e^\alpha  = \begin{cases}
\mathrlap{\tfrac{1}{4}e^\alpha} & \quad \mbox{if } i \in \supp(\alpha) \\
0 & \quad \mbox{otherwise} \end{cases}\\
e^\alpha e^\beta &= \begin{cases}
e^\alpha_{\phantom{\alpha}\vert \alpha \cap \beta \vert -1} e^\beta  = \lambda e^{\alpha+\beta} & \mbox{if } \alpha \neq \beta, \beta^c \\
\mathrlap{e^\alpha_{\phantom{\alpha}\vert \alpha \vert -3} e^\alpha} \phantom{e^\alpha_{\phantom{\alpha}\vert \alpha \cap \beta \vert -1} e^\beta} = 4\lambda^2 \sum_{\supp(\alpha)} t_i & \mbox{if } \beta = \alpha \\
\mathrlap{q^\alpha_{\phantom{\alpha}1} q^{\alpha^c}} \phantom{e^\alpha_{\phantom{\alpha}\vert \alpha \cap \beta \vert -1} e^\beta} = 0 & \mbox{if } \beta = \alpha^c
\end{cases}
\end{align*}
Then $A$ is a finite-dimensional algebra which lies in
\[
V_2 \oplus \bigoplus_{w \in \mbox{wt}(C)} V_{\frac{w}{2}}
\]
\end{theorem}
\begin{proof}
This follows from Proposition \ref{VOAprod}.
\end{proof}
Note that the algebra $A$ is contained in the weight $2$ part of the $M_C$ if and only if $C^*$ contains only weight $4$ codewords.

\end{document}